\documentclass[11pt]{amsart}
\usepackage{graphicx}
\usepackage{amsmath}
\usepackage{amssymb}
\usepackage{amsfonts}
\usepackage{tikz}
\usetikzlibrary{tqft}
\usetikzlibrary{shapes}
\usepackage{cite}
\usepackage[margin=1.40in]{geometry}
 \newtheorem{theorem}{Theorem}[subsection]
 \newtheorem{corollary}[theorem]{Corollary}

 \theoremstyle{definition}
 \newtheorem{definition}[theorem]{Definition}
 \theoremstyle{remark}
 \numberwithin{equation}{subsection}
\begin{document}

\title{Reidemeister-Franz torsion of Compact Orientable Surfaces via Pants Decomposition}
\author{ESMA D\.IR\.ICAN ERDAL and YA\c{S}AR S\"OZEN}

\address[1]{%
Deptartment of Mathematics\\
\.{I}zmir Institute of Technology\\
 35430, \.{I}zmir\\
Turkey}
\email{esmadirican@iyte.edu.tr}

\address[2]{%
Deptartment of Mathematics\\
Hacettepe University\\
06800, Ankara\\
Turkey}
\email{ysozen@hacettepe.edu.tr}

\thanks{Partially supported by T\"{U}B\.{I}TAK under the
project number 114F516. The first author would also like to thank T\"{U}B\.{I}TAK for the financial support. }

\subjclass[2010]{Primary 55U99; Secondary (optional) 18G99, 57Q10}
\keywords{Reidemeister-Franz torsion, compact orientable surfaces, pair of pants, period matrix}

\begin{abstract}
	Let $\Sigma_{g,n}$ denote the compact orientable surface with genus $g\geq 2$ and boundary disjoint union of $n$ circles. By using a particular pants-decomposition of $\Sigma_{g,n},$ we obtain a formula that computes the Reidemeister-Franz torsion of $\Sigma_{g,n}$ in terms of Reidemeister torsions of pairs of pants. 
\end{abstract}	
\maketitle

\section{Introduction}
The Reidemeister-Franz torsion (or R-torsion) was introduced by Reidemeister to classify $3$ dimensional lens spaces \cite{reidemeister}. This invariant was later generalized by Franz to other dimensions \cite{franz} and shown to be a topological invariant by Kirby-Siebenmann \cite{Kirby}. The R-torsion is also an invariant of the basis of the homology of a manifold \cite{milnor}. Moreover, for compact orientable Riemannian manifolds the R-torsion is equal to the analytic torsion \cite{Cheeger}.\\

Using the combinatorial definition of the Reidemeister torsion, Witten computed the volume of the moduli 
 space $\mathcal{M}$ of gauge equivalence classes of flat connections on a compact Riemann surface \cite{witten}. The combinatorial torsion is equivalent to the Ray-Singer analytic torsion \cite{Cheeger}. In the quantum field theory, one important ingredient was the ability to compute by decomposing a surface into elementary pieces. The pair of pants is a $(1+1)$-dimensional bordism, which corresponds to a product or 
 co-product (depending on its orientation) in a $2$-dimensional TQFT. Witten established a formula to compute the Ray-Singer analytic torsion of a pair of pants by using its cell decomposition. He also gave a cutting formula for orientable closed surface $\Sigma_{g,0}$ by decomposing an orientable surface $\Sigma_{g,0}$ of genus $g$ into $2g-2$ pairs of pants. \\
 
The present paper provides a formula to compute the Reidemeister-Franz torsion of a pair of pants in terms of the determinant of the period matrix of the Poincar\'{e} dual basis of $H^1(\Sigma_{2,0}).$ Then it expresses the  Reidemeister-Franz torsion of orientable compact surface $\Sigma_{g,n}$ as the product of the  Reidemeister-Franz torsions of pairs of pants.\\
 
For a manifold $M$ and an  integer $\eta$, we denote by 
$\mathbf{h}^{M}_\eta$ the basis of the homology $H_\eta(M)=H_\eta(M; \mathbb{R}).$ Note that $\Sigma_{2,0}$ is the double of $\Sigma_{0,3}.$ Let $\Delta_{0,2}(\Sigma_{2,0})$ be the matrix of the intersection pairing of $\Sigma_{2,0}$ in the bases
$\mathbf{h}^{\Sigma_{2,0}}_0,$ $\mathbf{h}^{\Sigma_{2,0}}_{2},$ 
and $\mathbf{h}_{\Sigma_{2,0}}^1=\{\omega_j\}_{1}^{4}$ denote the
Poincar\'{e} dual basis of $H^1(\Sigma_{2,0})$ corresponding to $\mathbf{h}^{\Sigma_{2,0}}_1$. We first prove the following theorem for the R-torsion of the pair of pants $\Sigma_{0,3}.$ 
\begin{theorem}\label{thm1}
	For a given basis $\mathbf{h}^{\Sigma_{0,3}}_i,$ $i=0,1,$ there is a basis
	$\mathbf{h}^{\Sigma_{2,0}}_\eta,$ $\eta=0,1,2$ such that the following formula holds
		\[|\mathbb{T}(\Sigma_{0,3},\{\mathbf{h}_i^{\Sigma_{0,3}}\}_{0}^1)|= \sqrt{\left|\frac{\det \Delta_{0,2}(\Sigma_{2,0})
			}{\det\wp(\mathbf{h}_{\Sigma_{2,0}}^1,\Gamma)}\right|},\]
where $\Gamma=\{\Gamma_1,\Gamma_2,\Gamma_{3},\Gamma_{4}\}$ is the canonical
basis for $H_1(\Sigma_{2,0}),$ i.e. $i=1,2,$ $\Gamma_i$ intersects
$\Gamma_{i+2}$ once positively and does not intersect others, and
$\wp(\mathbf{h}_{\Sigma_{2,0}}^1,\Gamma)= [\int_{\Gamma_i}\omega_j]$ is
the period matrix of $\mathbf{h}_{\Sigma_{2,0}}^1$ with respect to the basis
$\Gamma.$ 
\end{theorem}
 By using the pants decomposition of $\Sigma_{g,n}$ as in Figure 1, we prove the following theorem.
	\begin{theorem}
	\label{thm2}	
	If $\mathbf{h}_\eta^{\Sigma_{g,n}}$ is a given basis, $\eta=0,1,$ then for each $\nu=1,\ldots,2g-2+n$ there exists a basis $\mathbf{h}_\eta^{{\Sigma_{0,3}^{\nu}}}$ such that 
	\begin{eqnarray*}
		|\mathbb{T}(\Sigma_{g,n},\{\mathbf{h}_\eta^{\Sigma_{g,n}}\}_{0}^1)|=\prod_{\nu=1}^{2g-2+n}
		|\mathbb{T}({\Sigma_{0,3}^{\nu}},\{\mathbf{h}_\eta^{{\Sigma_{0,3}^{\nu}}}\}_{0}^1)|,
	\end{eqnarray*}
	where $\Sigma_{0,3}^{\nu}$ is the pair of pants in the decomposition labelled by $\nu.$
\end{theorem}

\section{R-torsion of a general chain complex}
Let 
 $C_{\ast}:0\to C_n\overset{\partial_n}{\rightarrow} %
	C_{n-1}\rightarrow\cdots\rightarrow C_1 \overset{%
		\partial_1}{\rightarrow} C_0\rightarrow 0$
	be a chain complex of 
	finite dimensional vector spaces over $\mathbb{R}.$ Let $B_p(C_{\ast})=\mathrm{Im}\partial_{p+1},$ 
	$Z_p(C_{\ast})=\mathrm{Ker}\partial_{p},$ and $H_p(C_{\ast})=Z_p(C_{\ast})/B_p(C_{\ast})$ denote the 
	$p$-th homology of the
	chain complex $C_{\ast}$ for $p=0,\ldots,n.$ 
	Then we have the following short exact sequences 
	\begin{equation}\label{eq3}
	0\to Z_p(C_\ast) \overset{\imath}{\rightarrow} C_p(C_\ast)
	\overset{\partial_p}{\rightarrow} B_{p-1}(C_\ast) \to 0,
	\end{equation}
	\begin{equation}\label{eq4}
	0\to B_p(C_\ast) \overset{\imath}{\rightarrow} Z_p(C_\ast)
	\overset{\varphi_p}{\to} H_p(C_\ast) \to 0.
	\end{equation}
	Here, $\imath$ and $\varphi_p$ are the inclusion and the natural
	projection, respectively. If we apply the Splitting Lemma to the above 
	short exact sequences, then $C_p(C_\ast)$ can be expressed as the following direct sum
	$$B_{p}(C_\ast) \oplus \ell_p(H_p(C_{\ast }))\oplus s_p(B_{p-1}(C_\ast)).$$ Let $\mathbf{c_p},$ 
	$\mathbf{b_p},$ 
	and $\mathbf{h_p}$ be respectively bases of
	$C_p(C_\ast),$ $B_p(C_\ast),$ and $H_p(C_\ast).$
	Then we obtain a new basis $\mathbf{b}_p\sqcup
	\ell_p(\mathbf{h}_p)\sqcup s_p(\mathbf{b}_{p-1})$ for $C_p(C_\ast).$
\begin{definition}
	 The \emph{R-torsion} of $C_{\ast}$ with respect to bases $\{\mathbf{c}_p\}_{0}^n,$ $\{\mathbf{h}_p\}_{0}^{n}$ is defined by 
	$$\mathbb{T}\left(C_{\ast},\{\mathbf{c}_p\}_{0}^n,\{\mathbf{h}_p\}_{0}^n \right)
	=\prod_{p=0}^n \left[\mathbf{b}_p\sqcup \ell_p(\mathbf{h}_p)\sqcup
	s_p(\mathbf{b}_{p-1}), \mathbf{c}_p\right]^{(-1)^{(p+1)}}.$$ Here,
	$\left[\mathbf{b}_p\sqcup \ell_p(\mathbf{h}_p)\sqcup
	s_p(\mathbf{b}_{p-1}), \mathbf{c}_p\right]$ is the determinant of
	the change-base-matrix from basis $\mathbf{c}_p$ to $\mathbf{b}_p\sqcup \ell_p(\mathbf{h}_p)\sqcup s_p(\mathbf{b}_{p-1})$ of $C_{p}(C_\ast).$ 
\end{definition}	
The R-torsion of a general  chain complex $C_{\ast}$ is an element of the dual of the vector space $$\bigotimes_{p=0}^n(\det H_p(C_{\ast}))^{(-1)^p},$$ see \cite[pp. 185]{witten} and \cite[Thm. 2.0.6]{sozen}.

For a smooth $m$-manifold $M$ with a cell decomposition $K,$ there is a chain complex 
$$C_{\ast}(K): 0\to C_m(K)\stackrel{\partial_m}{\rightarrow}C_{m-1}(K)\rightarrow \cdots\rightarrow C_1(K)\stackrel{\partial_1}{\rightarrow} C_0(K)\to 0,$$ where $\partial_i$ is the usual boundary operator. The R-torsion of $M$ is defined as the R-torsion of its cellular chain complex $C_{\ast}(K)$ in the bases $\{\mathbf{c}_i\}_0^m$ and $\{\mathbf{h}_i\}_0^m.$ Here, $\mathbf{c}_i$ is the geometric basis for the $i$-cells $C_i(K),$ $i=0,\ldots,m.$ By \cite[Lem. 2.0.5]{sozen}, the R-torsion of $M$ does not depend on the cell decomposition $K.$ Thus, we write $\mathbb{T}(M,\{\mathbf{h}_i\}_{0}^m)$	instead of $\mathbb{T}(C_{\ast}(K),\{\mathbf{c}_i\}_{0}^m,\{\mathbf{h}_i\}_{0}^m).$ For further details we refer to \cite{sozen,sozen2,turaev}.
\begin{corollary}
\label{cor2}
		Let $Y={\mathbb{S}^1}\times [-\epsilon,+\epsilon]$ be a cylinder with 
		boundary circles ${\mathbb{S}^1}\times\{-\epsilon\}$ and 
		${\mathbb{S}^1}\times\{+\epsilon\},$ where $\epsilon>0.$ 
		Let $\mathbf{h}_i$ be a basis of $H_i(Y)$ for $i=0,1.$ 
		By K\"unneth formula, we have the isomorphisms: 
		$C_i(Y)\overset{\varphi_i}{\cong} C_i(\mathbb{S}^1)$ and  
		$H_i(Y)\overset{[\varphi_i]}{\cong} H_i(\mathbb{S}^1).$ Then \cite[Thm. 3.5]{sozen2} gives the following result
		$$
		\left|\mathbb{T}\left(Y,\left\{\mathbf{h}_0,\mathbf{h}_1\right\}\right) \right|=\left|\mathbb{T}\left(\mathbb{S}^1
		,\left\{[\varphi_0](\mathbf{h}_0),
		[\varphi_1](\mathbf{h}_1)\right\}\right) \right|=1.$$ 
	\end{corollary}

\section{Proofs of main results}
\begin{proof}[Proof of Theorem \ref{thm1}]
For any manifold $M$, let $C_*(M)$ denote the associated cellular chain complex. 
\begin{figure}[h]
		\begin{center}
			%\begin{tikzpicture}[scale=2.5/3,rotate=0,every tqft/.style={transform shape}]
			\begin{tikzpicture}[rotate=180,every tqft/.style={transform shape}]
\pic[
tqft,
incoming boundary components=0,
outgoing boundary components=3,
draw,
name=a
];
\pic[
tqft,
incoming boundary components=3,
outgoing boundary components=0,
draw,
at=(a-outgoing boundary 1),
cobordism edge/.style={draw},every outgoing boundary component/.style={draw,black!50!black},every incoming boundary component/.style={draw,dashed,black!50!black},
name=b
];
\node at (-0.70,-2.0) {$S_3$};
			\node at (1.90,-1.55) {$S_2$};
			\node at (4.70,-2.0) {$S_1$};
\end{tikzpicture}
\end{center}
\caption{Double of the pair of pants $\Sigma_{0,3}.$}
\end{figure}
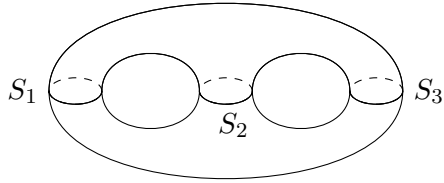

Note that $\Sigma_{2,0}$ is the double of $\Sigma_{0,3}.$ Let 
$\mathcal{B}$ be the intersection of the pairs of pants in 
$\Sigma_{2,0},$ so $\mathcal{B}$ is homeomorphic to the disjoint union of three circles, $\mathbb{S}_1 \amalg \mathbb{S}_2 \amalg\mathbb{S}_3.$ Then there is the natural short exact sequence of the chain complexes
\begin{equation}\label{eq2}
	0\to C_{\ast}(\mathcal{B})\rightarrow C_{\ast}(\Sigma_{0,3})\oplus C_{\ast}(\Sigma_{0,3})
	\rightarrow C_{\ast}(\Sigma_{2,0})\to 0
\end{equation} and the Mayer-Vietoris sequence associated to (\ref{eq2}) is
\begin{eqnarray}\label{eq1}
	&&\mathcal{H}_{\ast}:0\overset{\alpha}{\to  }
	H_2(\Sigma_{2,0})\overset{f}{ \to }
	H_1(\mathcal{B})\overset{g}{\to  } H_1(\Sigma_{0,3})\oplus H_1(\Sigma_{0,3}) \overset{h}{\to } H_1(\Sigma_{2,0}) \\
	& & \quad \quad \quad\overset{i}{ \to  }  H_0(\mathcal{B})\overset{j}{ \to } H_0(\Sigma_{0,3})\oplus H_0(\Sigma_{0,3}) \overset{k}{\to  } H_0(\Sigma_{2,0})\overset{\ell}{\to  } 0. \nonumber
\end{eqnarray}
 Let us denote
by $C_p(\mathcal{H}_{\ast})$ the vector spaces in (\ref{eq1}) for $p=0,\ldots,6$ and consider the short exact sequences (\ref{eq3}) and (\ref{eq4}) for $\mathcal{H}_{\ast}.$ Let us take the isomorphism $s_{p}:B_{p-1}(\mathcal{H}_{\ast}) \rightarrow s_{p}(B_{p-1}(\mathcal{H}_{\ast}))$ obtained by the First Isomorphism Theorem as a section of $C_p(\mathcal{H}_{\ast})\rightarrow B_{p-1}(\mathcal{H}_{\ast})$ for each $p.$ By the exactness of $\mathcal{H}_{\ast},$ we get 
$Z_p(\mathcal{H}_{\ast})=B_p(\mathcal{H}_{\ast}).$ Applying the Splitting Lemma to (\ref{eq4}), we have
\begin{equation}\label{eq5A}
C_p(\mathcal{H}_{\ast})=B_p(\mathcal{H}_{\ast})\oplus s_{_{p}}(B_{p-1}(\mathcal{H}_{\ast})).
\end{equation}
 Then the R-torsion of $\mathcal{H}_{\ast}$
with respect to basis $\{\mathbf{h}_p\}_{0}^n$ is given as follows
\begin{equation*}
\mathbb{T}\left( \mathcal{H}_{\ast},\{\mathbf{h}_p\}_{0}^n,\{0\}_{0}^n\right)
 =\prod_{p=0}^n [\mathbf{h}'_p,\mathbf{h}_p]^{(-1)^{(p+1)}},
\end{equation*}
where $\mathbf{h}'_p=\mathbf{b}_p\sqcup
s_p(\mathbf{b}_{p-1})$ for each $p.$ In \cite{milnor}, Milnor proved that the R-torsion does not depend on bases $\mathbf{b}_p$ and
sections $s_p, \ell_p.$ Therefore, we will
choose a suitable bases $\mathbf{b}_p$ and sections $s_p$ so that
$\mathbb{T}( \mathcal{H}_{\ast},\{\mathbf{h}_p\}_{0}^n,\{0\}_{0}^n)=1.$ \\

Let us consider the space $C_0(\mathcal{H}_{\ast})=H_0(\Sigma_{2,0})$ in (\ref{eq5A}). Then $\mathrm{Im}(\ell)=0$ yields
\begin{equation}\label{eq8}
C_0(\mathcal{H}_{\ast})=\mathrm{Im}(k)\oplus
s_{_0}(\mathrm{Im}(\ell))=\mathrm{Im}(k).
\end{equation}
Since $\{(\mathbf{h}_0^{\Sigma_{0,3}},0),(0,\mathbf{h}_0^{\Sigma_{0,3}})\}$ is the basis of $H_0(\Sigma_{0,3})\oplus H_0(\Sigma_{0,3}),$ 
$$\{\alpha_{_{11}}k(\mathbf{h}_0^{\Sigma_{0,3}},0)+\alpha_{_{12}}k(0,\mathbf{h}_0^{\Sigma_{0,3}})\}$$ can be taken as the basis $\mathbf{h}^{\mathrm{Im}(k)}$ of $\mathrm{Im}(k),$ where $(\alpha_{_{11}}, \alpha_{_{12}})$ is a non-zero vector. By (\ref{eq8}), $\mathbf{h}^{\mathrm{Im}(k)}$ becomes the obtained basis $\mathbf{h}'_0$ of $C_0(\mathcal{H}_{\ast}).$ If we take the initial basis $\mathbf{h}_0$
(namely, $\mathbf{h}_0^{\Sigma_{2,0}}$) of $C_0(\mathcal{H}_{\ast})$ as
$\mathbf{h}'_0,$ then 
\begin{equation}\label{eq9}
[\mathbf{h}'_0,\mathbf{h}_0]=1.
\end{equation}

If we use (\ref{eq5A}) for
$C_1(\mathcal{H}_{\ast})=H_0(\Sigma_{0,3})\oplus H_0(\Sigma_{0,3}),$ then we get
\begin{equation}\label{eq11}
C_1(\mathcal{H}_{\ast})=\mathrm{Im}(j)\oplus s_{_1}(\mathrm{Im}(k)).
\end{equation}
Note that the given basis $\mathbf{h}_1$ of $C_1(\mathcal{H}_{\ast})$ is
$\{(\mathbf{h}_0^{\Sigma_{0,3}},0),(0,\mathbf{h}_0^{\Sigma_{0,3}})\}.$
Since $\mathrm{Im}(j)$ is $1$-dimensional subspace of $2$-dimensional space $C_1(\mathcal{H}_{\ast}),$ there is a non-zero vector
$(a_{_{21}},a_{_{22}})$ such that
$\{a_{_{21}}(\mathbf{h}_0^{\Sigma_{0,3}},0)+a_{_{22}}(0,\mathbf{h}_0^{\Sigma_{0,3}})\}$
 is a basis of $\mathrm{Im}(j).$ In the previous step, the basis of 
 $\mathrm{Im}(k)$ was chosen as $\mathbf{h}^{\mathrm{Im}(k)}$ so
  $$s_{_1}(\mathbf{h}^{\mathrm{Im}(k)})=a_{_{11}}(\mathbf{h}_0^{\Sigma_{0,3}},0)+a_{_{12}}(0,\mathbf{h}_0^{\Sigma_{0,3}}).$$ Then we obtain a non-singular $2\times 2$ matrix $A=[a_{ij}]$ with entries in $\mathbb{R}.$ Let us choose the basis of
$\mathrm{Im}(j)$ as
$$\mathbf{h}^{\mathrm{Im}(j)}=\{-(\det A)^{-1}[a_{_{21}}(\mathbf{h}_0^{\Sigma_{0,3}},0)+a_{_{22}}(0,\mathbf{h}_0^{\Sigma_{0,3}})]\}.$$
 By (\ref{eq11}), $\{\mathbf{h}^{\mathrm{Im}(j)},s_{_1}(\mathbf{h}^{\mathrm{Im}(k)})\}$ becomes
the obtained basis $\mathbf{h}'_1$ of $C_1(\mathcal{H}_{\ast}).$ Hence, we get
\begin{equation}\label{eq12}
[\mathbf{h}'_1,\mathbf{h}_1]=1.
\end{equation}

Considering (\ref{eq5A}) for $C_2(\mathcal{H}_{\ast})=H_0(\mathcal{B}),$ we obtain
\begin{equation}\label{eq14}
C_2(\mathcal{H}_{\ast})=\mathrm{Im}(i)\oplus s_{_2}(\mathrm{Im}(j)).
\end{equation}
Recall that $\{\mathbf{h}_0^{\mathbb{S}_1},\mathbf{h}_0^{\mathbb{S}_2},\mathbf{h}_0^{\mathbb{S}_3}\}$ is the given basis $\mathbf{h}_2$ of $C_2(\mathcal{H}_{\ast}).$ Since $\mathrm{Im}(i)$ and $s_{_2}(\mathrm{Im}(j))$ are $2$ and $1$-dimensional subspaces of $3$-dimensional space 
$C_2(\mathcal{H}_{\ast}),$ there are non-zero vectors 
$(b_{_{i1}},b_{_{i2}},b_{_{i3}}),$ $i=1,2,3$ such
that $\{\sum_{i=1}^{3}b_{{ji}}\mathbf{h}_0^{\mathbb{S}_i}\}_{j=1}^{2}$ is a basis of $\mathrm{Im}(i)$ and 
$$s_{_2}(\mathbf{h}^{\mathrm{Im}(j)})=\sum_{i=1}^{3}b_{_{3i}}\mathbf{h}_0^{\mathbb{S}_i}$$ is a basis of $s_{_2}(\mathrm{Im}(j)).$ Then $3\times 3$
 real matrix $B=[b_{_{ij}}]$  is invertible. Let us choose the basis of $\mathrm{Im}(i)$ as follows
 \begin{equation*}
 \mathbf{h}^{\mathrm{Im}(i)}=\left\{(\det B)^{-1} \sum_{i=1}^{3}b_{_{1i}}\mathbf{h}_0^{\mathbb{S}_i}, \;
\sum_{i=1}^{3}b_{_{2i}}\mathbf{h}_0^{\mathbb{S}_i}\right\}.
 \end{equation*}
 By (\ref{eq14}), $\{\mathbf{h}^{\mathrm{Im}(i)},
s_{_2}(\mathbf{h}^{\mathrm{Im}(j)})\}$ becomes the obtained basis $\mathbf{h}'_2$ of
$C_2(\mathcal{H}_{\ast})$ and we have
\begin{equation}\label{eq17}
[\mathbf{h}'_2,\mathbf{h}_2]=1.
\end{equation}

Using (\ref{eq5A}), $C_3(\mathcal{H}_{\ast})=H_1(\Sigma_{2,0})$ can be expressed as the following direct sum
\begin{equation}\label{eq19}
C_3(\mathcal{H}_{\ast})=\mathrm{Im}(h)\oplus s_{_3}(\mathrm{Im}(i)).
\end{equation}
Note that the basis of $H_1(\Sigma_{0,3})\oplus H_1(\Sigma_{0,3})$ is 
$$\{(\mathbf{h}_{1,1}^{\Sigma_{0,3}},0),(0,\mathbf{h}_{1,1}^{\Sigma_{0,3}}),
(\mathbf{h}_{1,2}^{\Sigma_{0,3}},0),(0,\mathbf{h}_{1,2}^{\Sigma_{0,3}})\}.$$ Since $\mathrm{Im}(h)$ is a $2$-dimensional space, we can choose the basis of
$\mathrm{Im}(h)$ as 
\begin{eqnarray*}
&&\mathbf{h}^{\mathrm{Im}(h)}=\left\{c_{_{11}}h(\mathbf{h}_{1,1}^{\Sigma_{0,3}},0)+c_{_{12}}h(0,\mathbf{h}_{1,1}^{\Sigma_{0,3}})+
c_{_{13}}h(\mathbf{h}_{1,2}^{\Sigma_{0,3}},0)+c_{_{14}}h(0,\mathbf{h}_{1,2}^{\Sigma_{0,3}}),
\right.\\
&& \quad \quad \quad \quad \;\;\left.  c_{_{21}}h(\mathbf{h}_{1,1}^{\Sigma_{0,3}},0)+c_{_{22}}h(0,\mathbf{h}_{1,1}^{\Sigma_{0,3}})+c_{_{23}}h(\mathbf{h}_{1,2}^{\Sigma_{0,3}},0)+c_{_{24}}h(0,\mathbf{h}_{1,2}^{\Sigma_{0,3}})\right\}.
\end{eqnarray*}
Here, $(c_{_{i1}},c_{_{i2}},c_{_{i3}},c_{_{i4}})$ is a non-zero vector 
for $i=1,2.$ Using (\ref{eq19}), we have that
$\left\{\mathbf{h}^{\mathrm{Im}(h)},
s_{_3}(\mathbf{h}^{\mathrm{Im}(i)})\right\}$ is the obtained basis $\mathbf{h}'_3$ of
$C_3(\mathcal{H}_{\ast}).$ If we take the initial basis
$\mathbf{h}_3$ (namely, $\mathbf{h}_1^{\Sigma_{2,0}}$) of
$C_3(\mathcal{H}_{\ast})$ as $\mathbf{h}'_3,$ then we get
\begin{equation}\label{eq20}
[\mathbf{h}'_3,\mathbf{h}_3] =1.
\end{equation}

If we consider (\ref{eq5A}) for $C_4(\mathcal{H}_{\ast})=H_1(\Sigma_{0,3})\oplus H_1(\Sigma_{0,3}),$ then we obtain
\begin{equation}\label{eq22}
C_4(\mathcal{H}_{\ast})=\mathrm{Im}(g)\oplus s_{_4}(\mathrm{Im}(h)).
\end{equation}
Recall that $\{(\mathbf{h}_{1,1}^{\Sigma_{0,3}},0),(0,\mathbf{h}_{1,1}^{\Sigma_{0,3}}),
(\mathbf{h}_{1,2}^{\Sigma_{0,3}},0),(0,\mathbf{h}_{1,2}^{\Sigma_{0,3}})\}$ is the given basis $\mathbf{h}_4$ of $C_4(\mathcal{H}_{\ast}).$ In the previous step, $\mathbf{h}^{\mathrm{Im}(h)}$ was chosen as the basis  of 
 $\mathrm{Im}(h)$ so
\begin{eqnarray*}
\nonumber&& {s_{_4}(\mathbf{h}^{\mathrm{Im}(h)})}
=\left\{c_{_{11}}(\mathbf{h}_{1,1}^{\Sigma_{0,3}},0)+c_{_{12}}(0,\mathbf{h}_{1,1}^{\Sigma_{0,3}})+c_{_{13}}(\mathbf{h}_{1,2}^{\Sigma_{0,3}},0)+c_{_{14}}(0,\mathbf{h}_{1,2}^{\Sigma_{0,3}}),\right.\\
& & \quad\quad\quad \quad\quad\quad\; \left.c_{_{21}}(\mathbf{h}_{1,1}^{\Sigma_{0,3}},0)+c_{_{22}}(0,\mathbf{h}_{1,1}^{\Sigma_{0,3}})+c_{_{23}}(\mathbf{h}_{1,2}^{\Sigma_{0,3}},0)+c_{_{24}}(0,\mathbf{h}_{1,2}^{\Sigma_{0,3}})\right\}
\end{eqnarray*}
is a basis of $s_{_4}(\mathrm{Im}(h)).$ As $\mathrm{Im}(g)$ is $2$-dimensional
subspace of $4$-dimensional space $C_4(\mathcal{H}_{\ast}),$ there are non-zero vectors
$(c_{_{i1}},c_{_{i2}},c_{_{i3}},c_{_{i4}}),$ $i=3,4$ such that
\begin{eqnarray*}
\nonumber&&\left\{c_{_{31}}(\mathbf{h}_{1,1}^{\Sigma_{0,3}},0)+c_{_{32}}(0,\mathbf{h}_{1,1}^{\Sigma_{0,3}})+c_{_{33}}(\mathbf{h}_{1,2}^{\Sigma_{0,3}},0)+c_{_{34}}(0,\mathbf{h}_{1,2}^{\Sigma_{0,3}}),\right.\\
& &  \left. \; \;  c_{_{41}}(\mathbf{h}_{1,1}^{\Sigma_{0,3}},0)+c_{_{42}}(0,\mathbf{h}_{1,1}^{\Sigma_{0,3}})+c_{_{43}}(\mathbf{h}_{1,2}^{\Sigma_{0,3}},0)+c_{_{44}}(0,\mathbf{h}_{1,2}^{\Sigma_{0,3}})\right\}
\end{eqnarray*}
is a basis of $\mathrm{Im}(g)$ and $C=[c_{_{ij}}]$ is the
 non-singular $4\times 4$ real matrix. Thus, we can choose the basis of $\mathrm{Im}(g)$ as  
\begin{eqnarray*}
\nonumber \mathbf{h}^{\mathrm{Im}(g)}=\left\{(\det C)^{-1}[c_{_{31}}(\mathbf{h}_{1,1}^{\Sigma_{0,3}},0)+c_{_{32}}(0,\mathbf{h}_{1,1}^{\Sigma_{0,3}})+c_{_{33}}(\mathbf{h}_{1,2}^{\Sigma_{0,3}},0)+c_{_{34}}(0,\mathbf{h}_{1,2}^{\Sigma_{0,3}})],\right.\\
\quad \quad \quad \quad \left. \; \; \; c_{_{41}}(\mathbf{h}_{1,1}^{\Sigma_{0,3}},0)+c_{_{42}}(0,\mathbf{h}_{1,1}^{\Sigma_{0,3}})+c_{_{43}}(\mathbf{h}_{1,2}^{\Sigma_{0,3}},0)+c_{_{44}}(0,\mathbf{h}_{1,2}^{\Sigma_{0,3}})\right\}.
\end{eqnarray*}
 By (\ref{eq22}), $\{\mathbf{h}^{\mathrm{Im}(g)},
{s_{_4}(\mathbf{h}^{\mathrm{Im}(h)})}\}$ becomes the obtained
basis $\mathbf{h}'_4$ of $C_4(\mathcal{H}_{\ast})$ and the following equation holds
\begin{equation}\label{eq25}
[\mathbf{h}'_4,\mathbf{h}_4]=1.
\end{equation}

Consider the space $C_5(\mathcal{H}_{\ast})=H_1(\mathcal{B}),$ then (\ref{eq5A}) becomes
\begin{equation}\label{eq27}
C_5(\mathcal{H}_{\ast})=\mathrm{Im}(f)\oplus s_{_5}(\mathrm{Im}(g)).
\end{equation}
Recall that the initial basis $\mathbf{h}_5$ of $C_5(\mathcal{H}_{\ast})$ is
$\{\mathbf{h}_1^{\mathbb{S}_1},\mathbf{h}_1^{\mathbb{S}_2},\mathbf{h}_1^{\mathbb{S}_3}\}.$ Since
$\mathrm{Im}(f)$ and $s_{_5}(\mathrm{Im}(g))$ are respectively $1$ and
$2$-dimensional subspaces of $3$-dimensional space $C_5(\mathcal{H}_{\ast}),$ there are non-zero vectors
$(d_{_{i1}},d_{_{i2}},d_{_{i3}}),$ $i=1,2,3$ such that $\{\sum_{i=1}^{3}d_{_{1i}}\mathbf{h}_1^{\mathbb{S}_i}\}$ is a  basis of $\mathrm{Im}(f)$ and 
$${s_{_5}(\mathbf{h}^{\mathrm{Im}(g)})} =\left\{ \sum_{i=1}^{3}d_{_{2i}}\mathbf{h}_1^{\mathbb{S}_i}, \; \sum_{i=1}^{3}d_{_{3i}}\mathbf{h}_1^{\mathbb{S}_i}\right\}$$
is a basis of $s_{_5}(\mathrm{Im}(g)).$ Then we get a non-singular $3\times 3$ real matrix $D=[d_{_{ij}}].$ Let us choose the basis of ${\mathrm{Im}(f)}$ as $$\mathbf{h}^{\mathrm{Im}(f)}=\left\{(\det D)^{-1} \sum_{i=1}^{3}d_{_{1i}}\mathbf{h}_1^{\mathbb{S}_i}\right\}.$$ By (\ref{eq27}), $\{\mathbf{h}^{\mathrm{Im}(f)},
{s_{_5}(\mathbf{h}^{\mathrm{Im}(g)})}\}$ becomes the obtained
basis $\mathbf{h}'_5$ of $C_5(\mathcal{H}_{\ast}).$ Hence, we obtain
\begin{equation}\label{eq30}
[\mathbf{h}'_5,\mathbf{h}_5]=1.
\end{equation}

Finally, let us consider $C_6(\mathcal{H}_{\ast})=H_2(\Sigma_{2,0})$. 
%Since  $\mathrm{Im}(f)$ is isomorphic to $C_6(\mathcal{H}_{\ast}),$ we can
%take the section $s_{_6}$ as $f^{-1}:\mathrm{Im}(f)\rightarrow C_6(\mathcal{H}_{\ast}).$ 
Since $\mathrm{Im}(\alpha)=0,$ (\ref{eq5A}) becomes
\begin{equation}\label{eq32}
C_6(\mathcal{H}_{\ast})=\mathrm{Im}(\alpha)\oplus
s_{_6}(\mathrm{Im}(f))=s_{_6}(\mathrm{Im}(f)).
\end{equation}
From (\ref{eq32}) it follows that
$s_{_6}(\mathbf{h}^{\mathrm{Im}(f)})$
is the obtained basis $\mathbf{h}'_6$ of $C_6(\mathcal{H}_{\ast}).$ If we take
the initial basis $\mathbf{h}_6$ (namely, $\mathbf{h}_2^{\Sigma_{2,0}}$) of
$C_6(\mathcal{H}_{\ast})$ as
$s_{_6}(\mathbf{h}^{\mathrm{Im}(f)}),$ then we have
\begin{equation}\label{eq33}
 [\mathbf{h}'_6,\mathbf{h}_6]=1.
 \end{equation}
 
If we combine (\ref{eq9}), (\ref{eq12}), (\ref{eq17}),
(\ref{eq20}), (\ref{eq25}), (\ref{eq30}), and (\ref{eq33}), then we get
\begin{equation}\label{eq34}
  \mathbb{T}(\mathcal{H}_{\ast},\{\mathbf{h}_p\}_{0}^6,\{0\}_{0}^6)
  =\prod_{p=0}^6 [\mathbf{h}'_p,\mathbf{h}_p]^{(-1)^{(p+1)}}=1.
\end{equation}
Since the natural bases in (\ref{eq2}) are
compatible, \cite[Thm. 3.2]{milnor} yields
\begin{equation}\label{eq35}
 \mathbb{T}(\Sigma_{0,3},\{\mathbf{h}_i^{\Sigma_{0,3}}\}_{0}^1)^2=
 \prod_{j=1}^3 \mathbb{T}(\mathbb{S}_j,\{\mathbf{h}_i^{\mathbb{S}_j}\}_{0}^1)
 \; \mathbb{T}(\Sigma_{2,0},\{\mathbf{h}_\eta^{\Sigma_{2,0}}\}_{0}^2) \;
\mathbb{T}(\mathcal{H}_{\ast},\{\mathbf{h}_p\}_{0}^6,\{0\}_{0}^6).
\end{equation}
Considering \cite[Thm. 3.5]{sozen2}, (\ref{eq34}), and (\ref{eq35}), we obtain
\begin{equation}\label{eq36}
|\mathbb{T}(\Sigma_{0,3},\{\mathbf{h}_i^{\Sigma_{0,3}}\}_{0}^1)|=\sqrt{|\mathbb{T}(\Sigma_{2,0},\{\mathbf{h}_\eta^{\Sigma_{2,0}}\}_{0}^2)|}.
\end{equation}
By Poincar\'{e} Duality, Theorem 4.1 in \cite{sozen2} and (\ref{eq36}), the main formula holds
 $$|\mathbb{T}(\Sigma_{0,3},\{\mathbf{h}_i^{\Sigma_{0,3}}\}_{0}^1)|= \sqrt{\left|\frac{\det \Delta_{0,2}(\Sigma_{2,0})
}{\det\wp(\mathbf{h}_{\Sigma_{2,0}}^1,\Gamma)}\right|}.$$
\end{proof}
A \emph{pants decomposition} of $\Sigma_{g,n}$ is a finite collection of disjoint smoothly embedded circles cutting $\Sigma_{g,n}$ into pairs of pants $\Sigma_{0,3}$ and tori with one boundary circle $\Sigma_{1,1}.$ The number of complementary components is 
$|\chi(\Sigma_{g,n})|=2g-2+n.$
\begin{figure}[h]
$$
			\begin{tikzpicture}[scale=0.55,rotate=60,every tqft/.style={transform shape}]
			%%%****1. PANTOLON*** %%%
			\pic[
			tqft,
			incoming boundary components=1,
			outgoing boundary components=2,
			offset=0,
			genus=0,
			hole 3/.style={},
			genus lower/.style={solid},
			cobordism edge/.style={draw},every outgoing boundary component/.style={draw,black!50!black},every incoming boundary component/.style={draw,black!50!black},
			between incoming and outgoing/.style={},
			between outgoing 2 and 3/.style={}, anchor={(-1,4)}
			];
			%%%% ****2.PANTOLON****%%%%
			\pic[
			tqft,
			incoming boundary components=1,
			outgoing boundary components=2,
			offset=-0,
			genus=0,
			hole 3/.style={},
			genus lower/.style={},
			cobordism edge/.style={draw},every outgoing boundary component/.style={draw,black!50!black},every incoming boundary component/.style={draw,black!50!black},
			between incoming and outgoing/.style={},
			between outgoing 2 and 3/.style={}, anchor={(0,5)}
			];
			%%%%%%***3. PANTOLON ***%%%%%
			\pic[
			tqft,
			incoming boundary components=1,
			outgoing boundary components=2,
			offset=-0,
			genus=0,
			hole 3/.style={},
			genus lower/.style={},
			cobordism edge/.style={draw},every outgoing boundary component/.style={draw,black!50!black},every incoming boundary component/.style={draw,black!50!black},
			between incoming and outgoing/.style={},
			between outgoing 2 and 3/.style={}, anchor={(-2,3)}
			];
			%black yerine blue yada red yazılabılır.
			
			%%%%%%%****** 1. TORUS****%%%%%%%
			\begin{scope}[scale=0.78,rotate=270,shift={(-14.7,2.52)}]
			\begin{scope}[scale=.95]
			\path[rounded corners=5pt] (-0.5,0)--(0,0.5)--(0.5,0) (-0.5,0)--(0,-0.55)--    (.9,0);
			\draw[rounded corners=11pt] (-1,.05)--(0,-.6)--(1,.05);
			\draw[rounded corners=8pt] (-.9,0)--(0,.7)--(.9,0);
			\end{scope}
			Cut 1
			\draw[black,-,rotate=90] (0.,0.85) arc (270:90:.08 and 0.310);
			\draw[densely dashed,black,rotate=90] (0.,0.85) 
			arc (-90:90:.10 and .310);
			Cut 2
			\draw (1.0605,0.5656) arc (45:315:1.5 and 0.8);
			\draw (1.0605,0.5656) to[out=-27.9,in=180] (2,.5);
			\draw (1.0605,-0.5656) to[out=27.9,in=180] (2,-.4);
			%	\node at (3.5,0) {$\textcolor{black}{x}$};
			%\node at (-0.4,-0.5) {$\textcolor{red}{\gamma_2}$};
			%Cut 3
			%\draw[] (2,-0.4) arc (270:90:.2 and 0.450);
			%\draw[] (2,-0.4) arc (-90:90:.2 and .450);
			\end{scope}
			%CUT İÇERDE ÇEMBER OLUSTURUR.
			
			%%%%***** 2.TORUS****%%%%%%
			\begin{scope}[scale=0.78,rotate=90,shift={(8.3,-2.6)}]
			\begin{scope}[scale=.95,rotate=180]
			\path[rounded corners=5pt] (-0.5,0)--(0,0.5)--(0.5,0) (-0.5,0)--(0,-0.55)--    (.9,0);
			\draw[rounded corners=11pt] (-1,.05)--(0,-.6)--(1,.05);
			\draw[rounded corners=8pt] (-.9,0)--(0,.7)--(.9,0);
			\end{scope}
			Cut 1
			\draw [densely dashed,black,rotate=90](0.,0.85) arc (270:90:.08 and 0.310);
			\draw[black,-,rotate=90] (0.,0.85) arc (-90:90:.10 and .310);
			%Cut 1
			%\draw[densely dashed,red] (0,-1.3) arc (270:90:.2 and 0.550);
			%\draw[red,->] (0,-1.3) arc (-90:90:.2 and .550);
			%Cut 2
			\draw (1.0605,0.5656) arc (45:315:1.5 and 0.8);
			\draw (1.0605,0.5656) to[out=-27.9,in=180] (2,.5);
			\draw (1.0605,-0.5656) to[out=27.9,in=180] (2,-.4);
			%\node at (2.2,1.0) {glue};
			%\node at (3.0,-3.6) {glue};
			%\node at (9.5,-3) {$x_0$};
			%	\node at (3.5,0) {$\textcolor{black}{x}$};
			%\node at (-0.4,-0.5) {$\textcolor{red}{\gamma_2}$};
			%Cut 3
			%\draw[] (2,-0.4) arc (270:90:.2 and 0.450);
			%\draw[] (2,-0.4) arc (-90:90:.2 and .450);
			\end{scope}
			
			%%%%% ***3.TORUS****%%%%%%
			%******shift için 1. koordinat yükseklik ikinci koordinat uzaklık******
			\begin{scope}[scale=0.78,rotate=90,shift={(5.64,-5.170)}]
			%***ALTTAKI SCOPE İCERDEKİ DELİĞİ OLUŞTURUR***
			\begin{scope}[scale=.95,rotate=180]
			\path[rounded corners=5pt] (-0.5,0)--(0,0.5)--(0.5,0) (-0.5,0)--(0,-0.55)--    (.9,0);
			\draw[rounded corners=11pt] (-1,.05)--(0,-.6)--(1,.05);
			\draw[rounded corners=8pt] (-.9,0)--(0,.7)--(.9,0);
			\end{scope}
			Cut 1
			\draw [densely dashed,black,rotate=90](0.,0.85) arc (270:90:.08 and 0.310);
			\draw[black,-,rotate=90] (0.,0.85) arc (-90:90:.10 and .310);
			%Cut 2
			\draw (1.0605,0.5656) arc (45:315:1.5 and 0.8);
			\draw (1.0605,0.5656) to[out=-27.9,in=180] (2,.5);
			\draw (1.0605,-0.5656) to[out=27.9,in=180] (2,-.4);
			%\node at (2.2,1.0) {glue};
			%\node at (3.0,-3.6) {glue};
			%\node at (9.5,-3) {$x_0$};
			%	\node at (3.5,0) {$\textcolor{black}{x}$};
			%\node at (-0.4,-0.5) {$\textcolor{red}{\gamma_2}$};
			%Cut 3
			%\draw[] (2,-0.4) arc (270:90:.2 and 0.450);
			%\draw[] (2,-0.4) arc (-90:90:.2 and .450);
			\end{scope}

			\node at (8,3) {$\ddots$};	
			
			%%%%%***** 4. TORUS***** %%%%%
			\begin{scope}[scale=0.78,rotate=90,shift={(3.1,-7.73)}]
			\begin{scope}[scale=.95,rotate=180]
			\path[rounded corners=5pt] (-0.5,0)--(0,0.5)--(0.5,0) (-0.5,0)--(0,-0.55)--    (.9,0);
			\draw[rounded corners=11pt] (-1,.05)--(0,-.6)--(1,.05);
			\draw[rounded corners=8pt] (-.9,0)--(0,.7)--(.9,0);
			\end{scope}
			Cut 1
			\draw [densely dashed,black,rotate=90](0.,0.85) arc (270:90:.08 and 0.310);
			\draw[black,-,rotate=90] (0.,0.85) arc (-90:90:.10 and .310);
			%Cut 2
			\draw (1.0605,0.5656) arc (45:315:1.5 and 0.8);
			\draw (1.0605,0.5656) to[out=-27.9,in=180] (2,.5);
			\draw (1.0605,-0.5656) to[out=27.9,in=180] (2,-.4);
			%\node at (2.2,1.0) {glue};
			%\node at (3.0,-3.6) {glue};
			%\node at (9.5,-3) {$x_0$};
			%	\node at (3.5,0) {$\textcolor{black}{x}$};
			%\node at (-0.4,-0.5) {$\textcolor{red}{\gamma_2}$};
			%Cut 3
			%\draw[] (2,-0.4) arc (270:90:.2 and 0.450);
			%\draw[] (2,-0.4) arc (-90:90:.2 and .450);
			\end{scope}
			
			%%%%*****5. TORUS***** %%%%%%
			\begin{scope}[scale=0.78,rotate=90,shift={(-1.8,-10.31)}]
			\begin{scope}[scale=.95,rotate=180]
			\path[rounded corners=5pt] (-0.5,0)--(0,0.5)--(0.5,0) (-0.5,0)--(0,-0.55)--    (.9,0);
			\draw[rounded corners=11pt] (-1,.05)--(0,-.6)--(1,.05);
			\draw[rounded corners=8pt] (-.9,0)--(0,.7)--(.9,0);
			\end{scope}
			Cut 1
			\draw [densely dashed,black,rotate=90](0.,0.85) arc (270:90:.08 and 0.310);
			\draw[black,-,rotate=90] (0.,0.85) arc (-90:90:.10 and .310);
			%Cut 1
			%\draw[densely dashed,red] (0,-1.3) arc (270:90:.2 and 0.550);
			%\draw[red,->] (0,-1.3) arc (-90:90:.2 and .550);
			%Cut 2
			\draw (1.0605,0.5656) arc (45:315:1.5 and 0.8);
			\draw (1.0605,0.5656) to[out=-27.9,in=180] (2,.5);
			\draw (1.0605,-0.5656) to[out=27.9,in=180] (2,-.4);
			%\node at (2.2,1.0) {glue};
			%\node at (3.0,-3.6) {glue};
			%\node at (9.5,-3) {$x_0$};
			%	\node at (3.5,0) {$\textcolor{black}{x}$};
			%\node at (-0.4,-0.5) {$\textcolor{red}{\gamma_2}$};
			%Cut 3
			%\draw[] (2,-0.4) arc (270:90:.2 and 0.450);
			%\draw[] (2,-0.4) arc (-90:90:.2 and .450);
			\end{scope}
			%Torusların yerini shift içindeki sayılarla oynayarak bulabiliriz.

			%%%%% ***4. PANTOLON*** %%%%%
			\begin{scope}[shift={(-1,-1)}]
			\pic[
			tqft,
			incoming boundary components=1,
			outgoing boundary components=2,
			offset=-0,
			genus=0,
			hole 3/.style={},
			genus lower/.style={},
			cobordism edge/.style={draw},every outgoing boundary component/.style={draw,black!50!black},every incoming boundary component/.style={draw,black!50!black},
			between incoming and outgoing/.style={},
			between outgoing 2 and 3/.style={}, anchor={(-3.5,1.6)}
			];
			\end{scope}
			
			%%%%% ***5. PANTOLON*** %%%%%
			\pic[
			tqft,
			incoming boundary components=1,
			outgoing boundary components=2,
			offset=-0,
			genus=0,
			hole 3/.style={},
			genus lower/.style={},
			cobordism edge/.style={draw},every outgoing boundary component/.style={draw,black!50!black},every incoming boundary component/.style={draw,black!50!black},
			between incoming and outgoing/.style={},
			between outgoing 2 and 3/.style={}, anchor={(-4.,0.10)}
			];
			
			%%%%% ***6. PANTOLON*** %%%%%
			\pic[
			tqft,
			incoming boundary components=1,
			outgoing boundary components=2,
			offset=-0,
			genus=0,
			hole 3/.style={},
			genus lower/.style={},
			cobordism edge/.style={draw},every outgoing boundary component/.style={draw,black!50!black},every incoming boundary component/.style={draw,black!50!black},
			between incoming and outgoing/.style={},
			between outgoing 2 and 3/.style={}, anchor={(-5.,-1.9)}
			];
			
			\node at (12,-2.8) {$\ddots$};
			
			%%%%% ***7. PANTOLON*** %%%%%
			%\pic[
			%tqft,
			%incoming boundary components=1,
			%outgoing boundary components=2,
			%offset=-0,
			%genus=0,
			%hole 3/.style={},
			%genus lower/.style={},
			%cobordism edge/.style={draw},every outgoing boundary component/.style={draw,black!50!black},every incoming boundary component/.style={draw,black!50!black},
			%between incoming and outgoing/.style={},
			%between outgoing 2 and 3/.style={}, anchor={(-6.,-2.9)}
			%];

			%%%%*****ETİKETLER*****%%%%%%
			
			%%(A,B): *A* yı arttırırsan yukarı *B* yi arttırırsan sola gider %%
			
			\node at (2.1,13.1){{\small $\mathbb{S}'_1$}}; 
			\node at (1.15,10.0){{\small $\mathbb{S}_1$}}; 
			
			\node at (1.15,8.0){{\small $\mathbb{S}_2$}}; 
			\node at (2.1,4.8){{\small $\mathbb{S}'_2$}}; 
			
			\node at (3.2,6.0){{\small $\mathbb{S}_3$}}; 
			\node at (4.,2.8){$\mathbb{S}'_3$};
			
			\node at (5.2,3.9){{\small $\mathbb{S}_4$}}; 
			\node at (6.,0.7){{\small $\mathbb{S}'_4$}};

			\node at (7.2,0.1){{\small $\mathbb{S}_g$}}; 
			\node at (8,-3.1){{\small $\mathbb{S}'_g$}};

			\node at (4.8,8.5){{\small $\mathbb{S}_{g+1}$}}; 
			\node at (6.8,6.5){{\small $\mathbb{S}_{g+2}$}}; 
			\node at (8.8,4.7){{\small $\mathbb{S}_{g+3}$}}; 
			
			\node at (8.8,2.7){{\small $\mathbb{S}_{2g-3}$}}; 
			\node at (10.8,0.7){{\small $\mathbb{S}_{2g-2}$}}; 
			\node at (12.8,-0.6){{\small $\mathbb{S}_{2g-1}$}}; 
			\node at (12.9,-3.1){{\small $\mathbb{S}_{2g+n-3}$}}; 
			%\node at (14.8,-5.){$\mathbb{S}_{2g+n-3}$}; 
			
			\node at (9.2,-2.1){{\small $\mathbb{S}_0$}}; 
			\node at (11.1,-6){{\small $\mathbb{S}_{-(n-1)}$}}; 
			%\node at (13.1,-7.4){$\mathbb{S}_{-(n-1)}$}; 
			\node at (14.9,-5.4){{\small $\mathbb{S}_{-n}$}}; 
			\end{tikzpicture}
		$$
		\caption{Compact orientable surface $\Sigma_{g,n}$ with genus 
			$g\geq 2$ and bordered by $n\geq 1$ circles. }
	\end{figure}
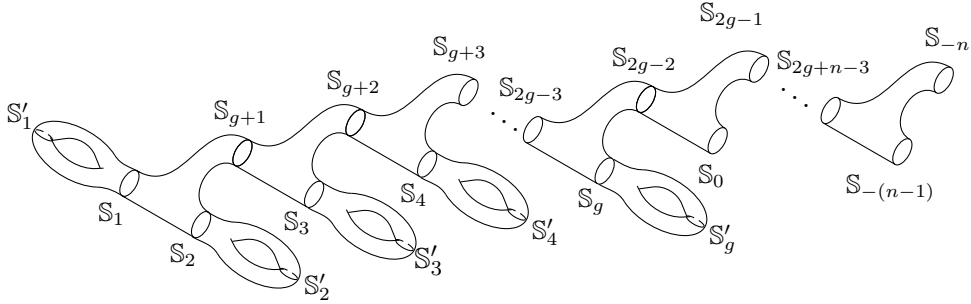
	
\begin{proof}[Proof of Theorem \ref{thm2}] 
Consider the decomposition of $\Sigma_{g,n},$ as in Figure 1, obtained by cutting the surface along the circles in the following order 
$$\mathbb{S}_1,\ldots,\mathbb{S}_g,\mathbb{S}_{ g+1},\ldots,\mathbb{S}_{2g-3+n}.$$ This decomposition consists of 
\begin{itemize}
\item[--]{the torus ${\Sigma_{1,1}^\nu}$ with boundary circle $\mathbb{S}_{\nu },$ $\nu=1,\ldots,g,$}
\item[--]{the pair of pants ${\Sigma_{0,3}^{g+1}}$ with boundaries $\mathbb{S}_1, \mathbb{S}_2, \mathbb{S}_{ g+1},$}
\item[--]{the pair of pants 
${\Sigma_{0,3}^{\nu+g}}$ with boundaries $\mathbb{S}_{g+\nu}, \mathbb{S}_{\nu+1}, \mathbb{S}_{ g+\nu-1},$ $\nu=2,\ldots,g-1,$ }
\item[--]{ the pair of pants ${\Sigma_{0,3}^{\nu+g}}$ with boundaries
$\mathbb{S}_{g+\nu}, \mathbb{S}_{ g+\nu-1}, \mathbb{S}_{g-\nu},$ $\nu=g,\ldots,g+n-3,$}
\item[--]{the pair of pants ${\Sigma_{0,3}^{2g-2+n}}$ with boundaries
$\mathbb{S}_{2g+n-3}, \mathbb{S}_{-(n-1)},$ $\mathbb{S}_{-(n-2)}.$}
\end{itemize}
Consider also the decomposition ${\Sigma_{1,1}^\nu}=Y_\nu \cup_{_{\partial Y_\nu}}{\Sigma_{0,3}^{\nu}},$ $\nu=1,\ldots,g,$ 
where $Y_\nu$ is the cylinder $\mathbb{S}'_\nu\times [-\varepsilon,+\varepsilon]$ and ${\Sigma_{0,3}^{\nu}}$ is the pair of pants with boundaries 
$\mathbb{S}'_\nu\times \{ -\varepsilon\},$ $\mathbb{S}'_\nu\times \{\varepsilon\},$  $\mathbb{S}_\nu$  for sufficiently small $\varepsilon>0.$

\noindent \textbf{Case 1 :} Consider the decomposition $\Sigma_{0,3}{\cup}_{\mathbb{S}_1}\Sigma_{0,n-1}$ of $\Sigma_{0,n}$ for $n\geq 4,$ where $\Sigma_{0,3}$ and $\Sigma_{0,n-1}$ are glued along the common boundary circle $\mathbb{S}_1.$ Then there is a short exact sequence of the chain complexes 
$$0\to C_{\ast}(\mathbb{S}_1)\rightarrow C_{\ast}(\Sigma_{0,3})\oplus C_{\ast}(\Sigma_{0,n-1})\rightarrow C_{\ast}(\Sigma_{0,n})\to 0$$ and the corresponding Mayer-Vietoris sequence $\mathcal{H}_{\ast}.$ By using the arguments stated in the proof of Theorem \ref{thm1} for the given bases $\mathbf{h}_\eta^{\Sigma_{0,n}}$ and $\mathbf{h}_\eta^{\mathbb{S}_1},$ $\eta=0,1,$ there exist bases $\mathbf{h}_\eta^{\Sigma_{0,3}}$ and $\mathbf{h}_\eta^{\Sigma_{0,n-1}}$ such that the R-torsion of $\mathcal{H}_{\ast}$ in the corresponding bases is $1$ and 
		\begin{equation}\label{eq36A}
		\mathbb{T}(\Sigma_{0,n},\{\mathbf{h}_\eta^{\Sigma_{0,n}}\}_{0}^1)
		=\mathbb{T}(\Sigma_{0,3},\{\mathbf{h}_\eta^{\Sigma_{0,3}}\}_{0}^1)\;
		\mathbb{T}(\Sigma_{0,n-1},\{\mathbf{h}_\eta^{\Sigma_{0,n-1}}\}_{0}^1)\;\mathbb{T}(\mathbb{S}_1,\{\mathbf{h}_\eta^{\mathbb{S}_1}\}_{0}^1)^{-1}.
		\end{equation}
	By  \cite[Thm. 3.5]{sozen2} and (\ref{eq36A}), we obtain
	\begin{equation}\label{eq36B}
	|\mathbb{T}(\Sigma_{0,n},\{\mathbf{h}_\eta^{\Sigma_{0,n}}\}_{0}^1)|=|\mathbb{T}(\Sigma_{0,3},\{\mathbf{h}_\eta^{\Sigma_{0,3}}\}_{0}^1)|
		|\mathbb{T}(\Sigma_{0,n-1},\{\mathbf{h}_\eta^{\Sigma_{0,n-1}}\}_{0}^1)|.
		\end{equation}
Applying (\ref{eq36B}) inductively, we get	
$$|\mathbb{T}(\Sigma_{0,n},\{\mathbf{h}_\eta^{\Sigma_{0,n}}\}_{0}^1)|=\prod_{\nu=1}^{n-2}|\mathbb{T}({\Sigma_{0,3}^\nu},\{\mathbf{h}_\eta^{{\Sigma_{0,3}^\nu}}\}_{0}^1)|.$$
	
\noindent \textbf{Case 2 :} 
For the decomposition $\Sigma_{1,1}=Y \cup_{_{\partial Y}}{\Sigma_{0,3}},$ where $Y=\mathbb{S}'\times [-\varepsilon,+\varepsilon],$ $\partial Y=\mathbb{S}'\times\{-\epsilon\}\sqcup\mathbb{S}'\times\{+\epsilon\},$ and ${\Sigma_{0,3}}$ is the pair of pants with boundaries 
		$\mathbb{S}'\times \{ -\varepsilon\},$ $\mathbb{S}'\times \{\varepsilon\},$  $\mathbb{S}$ for sufficiently small $\varepsilon>0,$ 
		we have the following short exact sequence of the chain complexes 
	\begin{equation}\label{eq37}
	0\to C_{\ast}({\Sigma_{0,3}}\cap Y)\rightarrow C_{\ast}({\Sigma_{0,3}})\oplus C_{\ast}( Y
	)\rightarrow C_{\ast}(\Sigma_{1,1})\to 0
	\end{equation}
	and the corresponding Mayer-Vietoris sequence $\mathcal{H}_{\ast}.$
	%\begin{eqnarray}\label{eq38}
	%&\mathcal{H}_{\ast}:0&\rightarrow H_1(\gamma_2)\oplus H_1(\gamma_2)\overset{f}{\rightarrow} H_1({\Sigma_{0,3}})\oplus H_1(Y)\overset{g}{\rightarrow} H_1(\Sigma_{1,1})\\ \nonumber
	%& &\overset{h}{\rightarrow} H_0(\gamma_2)\oplus H_0(\gamma_2)\overset{i}{\rightarrow} H_0({\Sigma_{0,3}})\oplus H_0(Y)\overset{j}{\rightarrow} H_0(\Sigma_{1,1})\overset{k}{\rightarrow} 0.
	%\end{eqnarray}
	If we follow the arguments in the proof of Theorem \ref{thm1} for the given bases $\mathbf{h}^{\Sigma_{1,1}}_\eta$  and $\mathbf{h}^{\mathbb{S}'}_\eta,$ $\eta=0,1,$ then we get the bases $\mathbf{h}^{{\Sigma_{0,3}}}_\eta$ and $\mathbf{h}^{Y}_\eta$ such that the R-torsion of $\mathcal{H}_{\ast}$ in the corresponding bases equals to $1$ and the formula is valid
		\begin{equation*}
		\mathbb{T}(\Sigma_{1,1},\{\mathbf{h}_\eta^{\Sigma_{1,1}}\}_{0}^1)=\mathbb{T}({\Sigma_{0,3}},\{\mathbf{h}_\eta^{\Sigma_{0,3}}\}_{0}^1) \; \mathbb{T}(Y,\{\mathbf{h}_\eta^Y\}_{0}^1) \; 
		\mathbb{T}(\mathbb{S}',\{\mathbf{h}_\eta^{\mathbb{S}'}\}_{0}^1)^{-2}.
		\end{equation*}
	 From  \cite[Thm. 3.5]{sozen2} and Corollary \ref{cor2} it follows
		\begin{equation*}
		|\mathbb{T}({\Sigma_{1,1}},\{\mathbf{h}_\eta^{\Sigma_{1,1}}\}_{0}^1)|=|\mathbb{T}({\Sigma_{0,3}},\{\mathbf{h}_\eta^{\Sigma_{0,3}}\}_{0}^1)|.
		\end{equation*}
		
\noindent \textbf{Case 3 :}	
	Let $\Sigma_{g-1,1}{\cup}_{{\mathbb{S}_1}}{\Sigma_{1,1}}$ be the decomposition of $\Sigma_{g,0}$, $g\geq 2,$ where ${\Sigma_{1,1}}$ and $\Sigma_{g-1,1}$ are glued along the common boundary circle $\mathbb{S}_1.$ 
		By the decomposition, there exists the natural short exact sequence 
		$$0\to C_{\ast}(\mathbb{S}_1)\rightarrow C_{\ast}(\Sigma_{g-1,1})\oplus C_{\ast}({\Sigma_{1,1}})\rightarrow C_{\ast}(\Sigma_{g,0})\to 0$$ 
		and its corresponding Mayer-Vietoris sequence 
		\begin{eqnarray*}
			&&\mathcal{H}_{\ast}:0\rightarrow H_2(\Sigma_{g,0})\overset{\delta_2}{\rightarrow}H_1(\mathbb{S}_1)\overset{f}{\rightarrow}
			H_1(\Sigma_{g-1,1})\oplus H_1({\Sigma_{1,1}})\overset{g}{\rightarrow}H_1(\Sigma_{g,0})\\
			& & \quad \quad \quad \overset{\delta_1}{\rightarrow}H_0(\mathbb{S}_1) 
			\overset{i}{\rightarrow}H_0(\Sigma_{g-1,1})\oplus H_0({\Sigma_{1,1}})\overset{j}{\rightarrow}H_0(\Sigma_{g,0})\overset{k}{\rightarrow}0.
		\end{eqnarray*}
		For the given bases $\mathbf{h}_\nu^{\Sigma_{g,0}}$ and $\mathbf{h}_\eta^{\mathbb{S}_1}$ with the condition 
		$\delta_2(\mathbf{h}_2^{\Sigma_{g,0}})=\mathbf{h}_1^{\mathbb{S}_1},$ $\nu=0,1,2,$ $\eta=0,1$, if we use the arguments stated in the proof of Theorem \ref{thm1}, then we obtain the bases $\mathbf{h}_\eta^{\Sigma_{g-1,1}}$ and $\mathbf{h}_\eta^{{\Sigma_{1,1}}}$ such that the R-torsion of $\mathcal{H}_{\ast}$ in the corresponding bases becomes $1$ and the following formula holds\\
		$$\mathbb{T}(\Sigma_{g,0},\{\mathbf{h}_\nu^{\Sigma_{g,0}}\}_{0}^2)=\mathbb{T}(\Sigma_{g-1,1},\{\mathbf{h}_\eta^{\Sigma_{g-1,1}}\}_{0}^1)\;
		\mathbb{T}({\Sigma_{1,1}},\{\mathbf{h}_\eta^{{\Sigma_{1,1}}}\}_{0}^1) \;\mathbb{T}(\mathbb{S}_1,\{\mathbf{h}_\eta^{\mathbb{S}_1}\}_{0}^1)^{-1}.$$
		By \cite[Thm. 3.5]{sozen2}, we obtain
$$|\mathbb{T}(\Sigma_{g,0},\{\mathbf{h}_\nu^{\Sigma_{g,0}}\}_{0}^2)|=|\mathbb{T}(\Sigma_{g-1,1},\{\mathbf{h}_\eta^{\Sigma_{g-1,1}}\}_{0}^1)|\;
		|\mathbb{T}({\Sigma_{1,1}},\{\mathbf{h}_\eta^{{\Sigma_{1,1}}}\}_{0}^1)|.$$
\textbf{Case 4 :}	
		Consider the decomposition $\Sigma_{g,n}=\Sigma_{g-1,n+1}\cup_{\mathbb{S}_1}{\Sigma_{1,1}}$ for $g\geq 2,$ $n\geq 1,$ where ${\Sigma_{1,1}}$ and $\Sigma_{g-1,n+1}$ are glued along the common boundary circle $\mathbb{S}_1.$ Then there is the natural short exact sequence of the chain complexes
		\begin{equation}\label{eq39}
		0\to C_{\ast}(\mathbb{S}_1)\rightarrow C_{\ast}(\Sigma_{g-1,n+1})\oplus C_{\ast}({\Sigma_{1,1}})\rightarrow C_{\ast}(\Sigma_{g,n})\to 0,
		\end{equation} 
		and the corresponding Mayer-Vietoris sequence $\mathcal{H}_{\ast}.$
	Using the arguments in the proof of Theorem \ref{thm1} for the given bases $\mathbf{h}_\eta^{\Sigma_{g,n}}$ and $\mathbf{h}_\eta^{\mathbb{S}_1},$ $\eta=0,1,$  we get the bases $\mathbf{h}_\eta^{\Sigma_{g-1,n+1}}$ and $\mathbf{h}_\eta^{{\Sigma_{1,1}}}$ such that the R-torsion of $\mathcal{H}_{\ast}$ in the corresponding bases is $1$ and 
	\begin{eqnarray*}
	 \mathbb{T}(\Sigma_{g,n},\{\mathbf{h}_\eta^{\Sigma_{g,n}}\}_{0}^1)=\mathbb{T}(\Sigma_{g-1,n+1},\{\mathbf{h}_\eta^{\Sigma_{g-1,n+1}}\}_{0}^1)\;
\mathbb{T}({\Sigma_{1,1}},\{\mathbf{h}_\eta^{{\Sigma_{1,1}}}\}_{0}^1)\;\mathbb{T}(\mathbb{S}_1,\{\mathbf{h}_\eta^{\mathbb{S}_1}\}_{0}^1)^{-1}.
	\end{eqnarray*}
	 By \cite[Thm. 3.5]{sozen2}, the R-torsion of $\Sigma_{g,n}$ satisfies the following formula
$$|\mathbb{T}(\Sigma_{g,n},\{\mathbf{h}_\eta^{\Sigma_{g,n}}\}_{0}^1)|=|\mathbb{T}(\Sigma_{g-1,n+1},\{\mathbf{h}_\eta^{\Sigma_{g-1,n+1}}\}_{0}^1)|\;	|\mathbb{T}({\Sigma_{1,1}},\{\mathbf{h}_\eta^{{\Sigma_{1,1}}}\}_{0}^1)|.$$

Applying the Cases 1-4 inductively, we have the following R-torsion formula for  the compact orientable surfaces $\Sigma_{g,n},$ $g\geq 2,$ $n\geq 0$
		$$|\mathbb{T}(\Sigma_{g,n},\{\mathbf{h}_\eta^{\Sigma_{g,n}}\}_{0}^1)|=\prod_{\nu=1}^{2g-2+n}|\mathbb{T}({\Sigma_{0,3}^\nu},\{\mathbf{h}_\eta^{{\Sigma_{0,3}^\nu}}\}_{0}^1)|.$$
\end{proof}
\section{Applications}
 \subsection{Compact 3-manifolds with boundary}  
		Let $N$ be a smooth compact orientable $3$-manifold whose boundary consists of finitely many closed orientable surfaces
		$\partial N=\Sigma_{g_{_1},0}\sqcup\Sigma_{g_{_2},0}\sqcup\cdots\sqcup\Sigma_{g_{_m},0}.$
		Let $d(N)$ be the double of $N.$ Consider the natural short exact sequence of the chain complexes 
		\begin{equation}\label{eq40}
		0\to C_{\ast}(\partial N)\rightarrow C_{\ast}(N)\oplus
		C_{\ast}(N)\rightarrow C_{\ast}(d(N))\to 0
		\end{equation}
	and the corresponding Mayer-Vietoris sequence $\mathcal{H}_{\ast}.$ For the given bases $\mathbf{h}_{\mu}^N,$  $\mathbf{h}_\nu^{\partial N},$ and $\mathbf{h}_{\mu}^{d(N)},$ $\nu=0,1,2,$ $\mu=0,1,2,3,$ we will denote the corresponding basis of 
	$\mathcal{H}_{\ast}$ by $\mathbf{h}_n,$ $n=0,\ldots,11.$ 
	As the bases in the sequence (\ref{eq40}) are 
	compatible, \cite[Thm. 3.2]{milnor}  yields
	\begin{equation}\label{eq42}
	\mathbb{T}(N,\{\mathbf{h}_{\mu}^N\}_{0}^3)^2=\mathbb{T}(\partial N,\{\mathbf{h}_\nu^{\partial N}\}_{0}^2)\; \mathbb{T}(d(N),\{\mathbf{h}_{\mu}^{d(N)}\}_{0}^3) \;\mathbb{T}(\mathcal{H}_{\ast},\{\mathbf{h}_n\}_{0}^{11}).
	\end{equation} 
By \cite[Thm. 3.5]{sozen2} and (\ref{eq42}), we have
	\begin{equation}\label{eq44}
	|\mathbb{T}(N,\{\mathbf{h}_{\mu}^N\}_{0}^3)|=\sqrt{|\mathbb{T}(\partial N,\{\mathbf{h}_\nu^{\partial N}\}_{0}^2)||\mathbb{T}(\mathcal{H}_{\ast},\{\mathbf{h}_n\}_{0}^{11})|}.
	\end{equation}
\noindent 	Note that $\partial N$ is equal to $\Sigma_{g_{_1},0}\sqcup\Sigma_{g_{_2},0}\sqcup\cdots\sqcup\Sigma_{g_{_m},0}.$ By \cite[Lem. 1.4]{sozen2}, we get
	\begin{equation}\label{eq45}
	|\mathbb{T}(\partial N,\{\mathbf{h}_\nu^{\partial N}\}_{0}^2)|= 
	\prod_{i=1}^{m}|\mathbb{T}(\Sigma_{g_{_i},0},\{\mathbf{h}_\nu^{\Sigma_{g_{_n},0}}\}_{0}^2)|.
	\end{equation}	
\noindent 	For each $i=1,\ldots,m,$ consider the given basis $\mathbf{h}_\nu^{\Sigma_{g_{_i},0}}$ for $\nu=0,1,2$ and pants-decompositions $\{{\Sigma_{0,3}^{j,i}}\}_{j=1}^{2g_{_i}-2}$ of $\Sigma_{g_{_i},0.}$ By using Theorem \ref{thm2}, we obtain the basis $\mathbf{h}_{\eta}^{{{\Sigma_{0,3}^{j,i}}}},$ ${\eta}=0,1,$ $j=1,\ldots,2g_{_i}-2$ such that
	\begin{equation}\label{eq46}
	|\mathbb{T}(\partial N,\{\mathbf{h}_\nu^{\partial N}\}_{0}^2)|= 
	\prod_{i=1}^{m}\prod_{j=1}^{2{g_i}-2}|\mathbb{T}({{\Sigma_{0,3}^{j,i}}},\{\mathbf{h}_{\eta}^{{\Sigma_{0,3}^{j,i}}}\}_0^1)|.
	\end{equation}
	\noindent Equations (\ref{eq45}) and (\ref{eq46}) yield the following formula
	\begin{equation*}
		|\mathbb{T}(N,\{\mathbf{h}_{\mu}^N\}_{0}^3)|=\sqrt{\prod_{i=1}^{m}\prod_{j=1}^{2{g_i}-2}|\mathbb{T}({{\Sigma_{0,3}^{j,i}}},\{\mathbf{h}_{\eta}^{{\Sigma_{0,3}^{j,i}}}\}_0^1)| \;| \mathbb{T}(\mathcal{H}_{\ast},\{\mathbf{h}_i\}_{0}^{11})|}.
		\end{equation*}
	\begin{corollary}
	\noindent Let $N$ be the handlebody of genus $g\geq 2.$ Clearly, the boundary $\partial N$ of $N$ is an orientable closed surface $\Sigma_{g,0}$ and the double $d(N)$ of $N$ is equal to $\underset{g}{\#}( {\mathbb{S}}\times {\mathbb{S}^2}).$ Then, we have the short exact sequence
		\begin{equation}\label{eq47}
		0\to C_{\ast}(\Sigma_{g,0})\rightarrow C_{\ast}(N)\oplus
		C_{\ast}(N)\rightarrow C_{\ast}(d(N))\to 0
		\end{equation}
		and the corresponding Mayer-Vietoris sequence 
		$\mathcal{H}_{\ast}.$ For the given bases $\mathbf{h}_{\mu}^{d(N)}$ and $\mathbf{h}_{\mu}^{N}$ $\mu=0,\ldots,3,$ following the arguments above, there exists a basis $\mathbf{h}_i^{\Sigma_{g,0}}$ $i=0,1,2$ such that in the corresponding bases the R-torsion of $\mathcal{H}_{\ast}$ is $1$ and from \cite[Thm. 3.5]{sozen2} it follows 
		\begin{equation*}
		|\mathbb{T}(N,\{\mathbf{h}_{\mu}^N\}_{0}^3)|=\sqrt{|\mathbb{T}(\Sigma_{g,0},\{\mathbf{h}_i^{\Sigma_{g,0}}\}_{0}^2)|}.
		\end{equation*}
	Let us consider the pants-decomposition 
		$\{{\Sigma_{0,3}^{j}}\}_{j=1}^{2g-2}$ of $\Sigma_{g,0}.$ By Theorem \ref{thm2}, there exists the basis $\mathbf{h}_\eta^{{{\Sigma_{0,3}^{j}}}}$
		for each $j=1,\ldots,2g-2,$ $\eta=0,1$ and the formula holds
	\begin{equation*}
		|\mathbb{T}(N,\{\mathbf{h}_\mu^N\}_{0}^3)|= \sqrt{
			\prod_{j=1}^{2g-2}|\mathbb{T}({{\Sigma_{0,3}^{j}}},\{\mathbf{h}_\eta^{{\Sigma_{0,3}^{j}}}\}_0^1)|}.
		\end{equation*}
\end{corollary}
	
	\subsection{Product of $2d$-manifolds and compact $3$-manifolds with boundary $\Sigma_{g,0}$ }
		Let $M$ be a smooth closed orientable $2d$-manifold ($d\geq1$) and $N$  an  smooth compact orientable $3$-manifold whose boundary consists of closed 
		orientable surface $\Sigma_{g,0}$ $(g\geq2).$ 
		Let $X$ be the product manifold $M\times N$ and $d(X)$ denote the double of $X.$ Clearly, the boundary of $X$ is $M\times\Sigma_{g,0}.$ Consider the natural short exact sequence of the chain complexes
		\begin{equation}\label{eq49}
		0\to C_{\ast}(M\times\Sigma_{g,0})\rightarrow C_{\ast}(X)\oplus
		C_{\ast}(X)\rightarrow C_{\ast}(d(X))\to 0
		\end{equation}
		\noindent and the Mayer-Vietoris sequence $\mathcal{H}_{\ast}$ corresponding to (\ref{eq49}). Let $\mathbf{h}_i^X,$ $\mathbf{h}_i^{d(X)},$ $\mathbf{h}_k^{M},$ and $\mathbf{h}_\ell^{\Sigma_{g,0}}$ be given bases for $i=0,\cdots,2d+3,$ $k=0,\ldots,2d,$ $\ell=0,1,2.$ Let
		$\mathbf{h}_\nu^{M\times\Sigma_{g,0}}$ denote the basis $\underset{i}{\oplus}\mathbf{h}_i^{M}\otimes\mathbf{h}_{\nu-i}^{\Sigma_{g,0}}$ of $H_\nu(M\times\Sigma_{g,0}),$ $\nu=0,\ldots,2d+2.$ For $n=0,\ldots,6d+11,$ let $\mathbf{h}_n$ be the corresponding basis of $\mathcal{H}_{\ast}.$ 
		Let $\{{\Sigma_{0,3}^{j}}\}_{j=1}^{2g-2}$ be the pants-decomposition of $\Sigma_{g,0}.$ Since the bases in the sequence (\ref{eq49}) are compatible and \cite[Lem. 1.4]{sozen2}, we obtain
	\begin{eqnarray}\label{eq50}
	&& \nonumber \mathbb{T}(X,\{\mathbf{h}_i^X \}_{0}^{2d+3})^2=\mathbb{T}(M\times\Sigma_{g,0},\{\mathbf{h}_\nu^{M\times\Sigma_{g,0}}\}_{0}^{2d+2})\; \mathbb{T}(d(X),\{\mathbf{h}_i^{d(X)}\}_{0}^{2d+3}) \\
	&& \quad \quad \quad \quad \quad \quad \quad \quad \;\; \; \times\; \mathbb{T}(\mathcal{H}_{\ast},\{\mathbf{h}_n\}_{0}^{6d+11}).
	\end{eqnarray} 
\noindent From \cite[Thm. 3.5]{sozen2} and (\ref{eq50}) it follows that
	\begin{equation}\label{eq52}
	|\mathbb{T}(X,\{\mathbf{h}_i^X\}_{0}^{2d+3})|=|\mathbb{T}(M\times\Sigma_{g,0},\{\mathbf{h}_\nu^{M\times\Sigma_{g,0}} \}_{0}^{2d+2})|^{1/2} \;|\mathbb{T}(\mathcal{H}_{\ast},\{\mathbf{h}_n\}_{0}^{6d+11})|^{1/2}.
	\end{equation}
	\noindent By  \cite[Thm. 3.1]{ozelszn}, the R-torsion of $M\times\Sigma_{g,0}$ satisfies the equality
	\begin{equation}\label{eq53}
	|\mathbb{T}(M\times\Sigma_{g,0},\{\mathbf{h}_\nu^{M\times\Sigma_{g,0}}\}_{0}^{2d+2})|=|{\mathbb{T}(M,\{\mathbf{h}_k^M\}_{0}^{2d
		})}|^{\chi(\Sigma_{g,0})} \; |{\mathbb{T}(\Sigma_{g,0},\{\mathbf{h}_\ell^{\Sigma_{g,0}}\}_{0}^{2})}|^{\chi(M)}.
	\end{equation}
	Here, $\chi$ is the Euler characteristic. 
	Then equations (\ref{eq52}) and (\ref{eq53}) yield	
		\begin{eqnarray}\label{eq54}
	&& \nonumber|\mathbb{T}(X,\{\mathbf{h}_i^X\}_{0}^{2d+3})|=|{\mathbb{T}(M,\{\mathbf{h}_k^M\}_{0}^{2d
		})}|^{\chi(\Sigma_{g,0})/2} \;  |{\mathbb{T}(\Sigma_{g,0},\{\mathbf{h}_\ell^{\Sigma_{g,0}}\}_{0}^{2})}|^{\chi(M)/2} \\
	&&\quad \quad \quad  \quad \quad \quad  \quad \quad  \quad \times \;	 {|\mathbb{T}(\mathcal{H}_{\ast},\{\mathbf{h}_n\}_{0}^{6d+11})|}^{1/2}.
	\end{eqnarray}	
	\noindent Since $\{{\Sigma_{0,3}^{j}}\}_{j=1}^{2g-2}$ is the pants-decomposition of $\Sigma_{g,0}$ as in Theorem \ref{thm2}, there exists a basis $\mathbf{h}_\eta^{\Sigma_{0,3}^j}$ of $H_\eta(\Sigma_{0,3}^j),$ $j=1,\ldots,2g-2,$ $\eta=0,1$ so that 
	\begin{eqnarray} \label{eq55}
	|\mathbb{T}(\Sigma_{g,0},\{\mathbf{h}_\ell^{\Sigma_{g,0}}\}_{0}^2)|=\prod_{j=1}^{2g-2}|\mathbb{T}(\Sigma_{0,3}^j,\{\mathbf{h}_\eta^{\Sigma_{0,3}^j}\}_{0}^1)|.
	\end{eqnarray}
\noindent Equations (\ref{eq54}) and (\ref{eq55}) yield
	\begin{eqnarray*}
	  &&	|\mathbb{T}(X,\{\mathbf{h}_i^X\}_{0}^{2d+3})|=\prod_{j=1}^{2g-2}{|\mathbb{T}(\Sigma_{0,3}^j,\{\mathbf{h}_\eta^{\Sigma_{0,3}^j}\}_{0}^1)|}^{\frac{\chi(M)}{2}}\;{|\mathbb{T}(M,\{\mathbf{h}_k^M\}_{0}^{2d
		})|}^{\frac{\chi(\Sigma_{g,0})}{2}} \\
				&& \quad \quad  \quad\quad \quad \quad  \quad \quad  \quad \times \; {|\mathbb{T}(\mathcal{H}_{\ast},\{\mathbf{h}_n\}_{0}^{6d+1})|}^{1/2}.
		\end{eqnarray*}
	
	\section*{Acknowledgment}
\noindent Theorem \ref{thm2} and Section 4 were proven in the first author's MSc thesis.

\bibliographystyle{plain}
\bibliography{rtorsion}

\end{document}